\documentclass[11pt,leqno]{amsart}

\usepackage{hyperref}

\newtheorem{theorem}{Theorem}[section]
\newtheorem{lemma}[theorem]{Lemma}

\theoremstyle{definition}

\title{Conciseness of coprime commutators in finite groups}
\author{Cristina Acciarri}

\address{Cristina Acciarri:  Department of Mathematics, University of Brasilia,
Brasilia-DF, 70910-900 Brazil}
\email{acciarricristina@yahoo.it}

\author{Pavel Shumyatsky} 

\address{Pavel Shumyatsky: Department of Mathematics, University of Brasilia,
Brasilia-DF, 70910-900 Brazil}

\email{pavel@unb.br}

\author{Anitha Thillaisundaram}

\address{Anitha Thillaisundaram: Institut f\"ur Algebra und Geometrie, Mathematische Fakult\"at, Otto-von-Guericke-Universit\"at Magdeburg, 39016 Magdeburg, Germany}

\email{anitha.t@cantab.net}

\keywords{commutators, concise words}

\subjclass[2010]{Primary  20D25;  Secondary 20F12}

\thanks{The research of the first and second authors was supported by CNPq-Brazil.}

\begin{document}
\begin{abstract} 
Let $G$ be a finite group. We show that the order of the subgroup generated by coprime $\gamma_k$-commutators (respectively $\delta_k$-commutators) is bounded in terms of the size of the set of coprime $\gamma_k$-commutators (respectively $\delta_k$-commutators). This is in parallel with the classical theorem due to Turner-Smith that the words $\gamma_k$ and $\delta_k$ are concise. 
\end{abstract}

\maketitle

\section{Introduction}
Let $w$ be a group-word in $n$ variables, and let $G$ be a group. The verbal subgroup $w(G)$ of $G$ determined by $w$ is the subgroup generated by the set $G_w$ consisting of all values $w(g_1,\ldots,g_n)$, where $g_1,\ldots,g_n$ are elements of $G$.  A word $w$ is said to be concise if whenever $G_w$ is finite for a group $G$, it always follows that $w(G)$ is finite. More generally, a word $w$ is said to be concise in a class of groups $\mathcal{X}$ if whenever $G_w$ is finite for a group $G\in \mathcal{X}$, it always follows that $w(G)$ is finite. In the sixties P. Hall asked whether every word is concise but  later Ivanov proved that this problem has a negative solution in its general form \cite{ivanov} (see also \cite[p.\ 439]{ols}). On the other hand, many relevant words are known to be concise. For instance, Turner-Smith \cite{TS} showed that the {\it lower central words} $\gamma_k$ and the {\it derived words} $\delta_k$ are concise; here the words $\gamma_k$ and $\delta_k$ are defined by the positions $\gamma_1=\delta_0=x_1$, $\gamma_{k+1}=[\gamma_k,x_{k+1}]$ and $\delta_{k+1}=[\delta_k,\delta_k]$. Wilson showed in \cite{jwilson} that the multilinear commutator words (outer commutator words) are concise. It has been proved by Merzlyakov \cite{merzlyakov} that every word is concise in the class of linear groups.

In \cite{brakrashu} a word $w$ was called boundedly concise in a class of groups $\mathcal{X}$ if for every integer $m$ there exists a number $\nu=\nu(\mathcal{X},w,m)$ such that whenever $|G_w|\leq m$ for a group $G\in \mathcal{X}$ it always follows that $|w(G)|\leq\nu$. Fern\'andez-Alcober and Morigi \cite{fernandez-morigi} showed that every word which is concise in the class of all groups is actually boundedly concise. Moreover they showed that whenever $w$ is a multilinear commutator word having at most $m$ values in a group $G$, one has $|w(G)|\leq (m-1)^{(m-1)}$. Questions on conciseness of words in the class of residually finite groups have been tackled in \cite{AS}. It was shown that if $w$ is a multilinear commutator word and  $q$ a prime-power, then the word $w^q$ is concise in the class of residually finite groups; and if $w=\gamma_{k}$ is the $k$th lower central word and $q$ a prime-power, then the word $w^{q}$ is boundedly concise in the class of residually finite groups.

The concept of (bounded) conciseness can actually be applied in a much wider context. Suppose $\mathcal{X}$ is a class of groups and $\phi(G)$ is a subset of $G$  for every group $G\in \mathcal{X}$. One can ask whether the subgroup generated by $\phi(G)$ is finite whenever $\phi(G)$ is finite. In the present paper we show bounded conciseness of coprime commutators in finite groups.

The coprime commutators $\gamma_k^*$ and $\delta_k^*$ have been introduced in \cite{forum} as a tool to study properties of finite groups that can be expressed in terms of commutators of elements of coprime orders. Let $G$ be a finite group. Every element of $G$ is a $\gamma_1^*$-commutator as well as a $\delta_0^*$-commutator.
Now let $k\geq 2$ and let $X$ be the set of all elements of $G$ that are powers of $\gamma_{k-1}^*$-commutators. An element $g$ is a $\gamma_k^*$-commutator if there exist $a\in X$ and $b\in G$ such that $g=[a,b]$ and $(|a|,|b|)=1$. For $k\geq 1$ let $Y$ be the set of all elements of $G$ that are powers of $\delta_{k-1}^*$-commutators. The element $g$ is a $\delta_k^*$-commutator if there exist $a,b\in Y$ such that $g=[a,b]$ and $(|a|,|b|)=1$. The subgroups of $G$ generated by all $\gamma_k^*$-commutators and all  $\delta_k^*$-commutators will be denoted by $\gamma_k^*(G)$ and $\delta_k^*(G)$, respectively. One can easily see that if $N$ is a normal subgroup of $G$ and $x$ an element whose image in $G/N$ is a $\gamma_k^*$-commutator (respectively a $\delta_k^*$-commutator), then there exists a $\gamma_k^*$-commutator $y\in G$ (respectively a $\delta_k^*$-commutator) such that $x\in yN$. It was shown in \cite{forum} that $\gamma_k^*(G)=1$ if and only if $G$ is nilpotent and $\delta_k^*(G)=1$ if and only if the Fitting height of $G$ is at most $k$. It follows that for any $k\geq2$ the subgroup $\gamma_k^*(G)$ is precisely the last term of the lower central series of $G$ (which is sometimes denoted by $\gamma_\infty(G)$) while for any $k\geq1$ the subgroup $\delta_k^*(G)$ is precisely the last term of the lower central series of $\delta_{k-1}^*(G)$. In the present paper we prove the following results.

\begin{theorem}\label{gut} Let $k\geq 1$ and $G$ a finite group in which the set of $\gamma_k^*$-commutators has size $m$. Then $|\gamma_k^*(G)|$ is $m$-bounded.
\end{theorem}

\begin{theorem}\label{gitis} Let $k\geq 0$ and $G$ a finite group in which the set of $\delta_k^*$-commutators has size $m$. Then $|\delta_k^*(G)|$ is $m$-bounded.
\end{theorem}

We remark that the bounds for $|\gamma_k^*(G)|$ and $|\delta_k^*(G)|$ in the above results do not depend on $k$. Thus, we observe here the phenomenon that in \cite{fernandez-morigi} was dubbed ``uniform conciseness". We make no attempts to provide explicit bounds for $|\gamma^*_k(G)|$ and $|\delta^*_k(G)|$ in Theorems \ref{gut} and \ref{gitis}. Throughout the paper we use the term $m$-bounded to mean that the bound is a function of $m$.

\section{Preliminaries}

We begin with a well-known result about coprime actions on finite groups. Recall that $[K,H]$ is the subgroup generated by $\{[k,h]:k\in K, h\in H\}$, and $[K,_iH]=[[K,_{i-1}H],H]$ for $i\ge2$.

\begin{lemma}[\cite{Isaacs}, Lemma 4.29]
\label{Lemma4.29}
Let $A$ act via automorphisms on $G$, where $A$ and $G$ are finite groups, and suppose that $(|G|,|A|)=1$. Then $[G,A,A]=[G,A]$.
\end{lemma}

For the following result from \cite{Pavel2}, recall that a subset $B$ of a group $A$ is \textit{normal} if $B$ is a union of conjugacy classes of $A$. 

\begin{lemma}
\label{Pavel}
Let $A$ be a group of automorphisms of a finite group $G$ with $(|A|,|G|)=1$. Suppose that $B$ is a normal subset of $A$ such that $A=\langle B \rangle$. Let $k\ge 1$ be an integer. Then $[G,A]$ is generated by the subgroups of the form $[G,b_1,\ldots,b_k]$, where $b_1,\ldots,b_k\in B$.
\end{lemma}

The following is an elementary property of $\delta^*_k$-commutators.

\begin{lemma}
\label{Lemma0}
Let $G$ be a finite group. For $k$ a non-negative integer, 
\[
\delta^*_k(\delta^*_1(G)) = \delta^*_{k+1}(G).
\]
\end{lemma}

\begin{proof}
We argue by induction. For $k=0$, the result is obvious by the definition of $\delta^*_0$-commutators.

Suppose the result holds for $k-1$. So
\[
\delta^*_{k-1}(\delta^*_1(G))=\delta^*_k(G).
\]

It was mentioned in the introduction that $\delta^*_{k+1}(G)=\gamma_\infty(\delta^*_k(G))$. By induction,
\[
\delta^*_{k+1}(G)=\gamma_\infty(\delta^*_{k-1}(\delta^*_1(G))),
\]
and viewing  $\delta^*_1(G)$ as the group in consideration, we have 
\[\gamma_\infty(\delta^*_{k-1}(\delta^*_1(G)))=\delta^*_k(\delta^*_1(G))
\]
 as required.
\end{proof}

Here is a helpful observation that we will use in both of our main results. Recall that a Hall subgroup of a finite group is a subgroup whose order is coprime to its index. Also, a finite group $G$ is metanilpotent if and only if $\gamma_\infty (G)$ is nilpotent.

\begin{lemma}
\label{Lemma1}
Let $G$ be a finite metanilpotent group and $P$ a Sylow $p$-subgroup of $\gamma_\infty(G)$, and let $H$ be a Hall $p'$-subgroup of $G$. Then $P=[P,H]$.
\end{lemma}

\begin{proof}
For simplicity, we write $K$ for $\gamma_\infty(G)$.
By passing to the quotient $G/\mathcal{O}_{p'}(G)$, we may assume that $P=K$. 

Let $P_1$ be a Sylow $p$-subgroup of $G$. So $G=P_1 H$. Now $P_1 /P$ is normal in $G/P$ as $G/P$ is nilpotent, but also $P \le P_1$; hence $P_1$ is normal in $G$. It follows that $K=[P_1,H]$, since in a nilpotent group all coprime elements commute. By Lemma \ref{Lemma4.29}, $[P_1,H,H]=[P_1,H]=P$, and so $P=[P_1,H]=[P,H]$.
\end{proof}

As it turns out, in the proofs of our main results we often reduce to the following case.
\begin{lemma}
\label{Lemma3}
Let $i$ and $m$ be positive integers. Let $P$ be an abelian $p$-group acted on by a $p'$-group $A$ such that 
\[
|\{[x,a_1,\ldots,a_i]:x\in P, a_1,\ldots,a_i\in A\}|=m.
\]
Then $|[P,_i A]|=2^m$, so is $m$-bounded.
\end{lemma}

\begin{proof}
We enumerate the set $\{[x,a_1,\ldots,a_i]:x\in P, a_1,\ldots,a_i \in A\}$ as $\{c_1,\ldots,c_m\}$.
As $P$ is abelian, we have that
\begin{equation*}
\label{eq:*}
[x,a_1,\ldots,a_i]^l =[x^l,a_1,\ldots,a_i]\quad\quad\quad (\dagger)
\end{equation*}
for all $x\in P, a_1,\ldots,a_i\in A$, and $l$ a positive integer.

Consider $g\in[P,_iA]$, which can be expressed as some product $c_1^{l_1} \ldots c_m^{l_m}$  for non-negative integers $l_1,\ldots,l_m$. We claim that $l_1,\ldots,l_m\in \{0,1\}$. For, if $l_j>1$ with $j\in \{1,\ldots,m\}$, we know from ($\dagger$) that $c_j^{l_j}\in \{c_1,\ldots,c_m\}$. We replace all such $c_j^{l_j}$ accordingly, so that $g$ is now expressed as $c_1^{k_1}\ldots c_m^{k_m}$  with $k_1,\ldots,k_m\in\{0,1\}$. Hence $|[P,_iA]|=2^m$.
\end{proof}

The well-known Focal Subgroup Theorem \cite{Rose} states that if $G$ is a finite group and $P$ a Sylow $p$-subgroup of $G$, then $P\cap G'$ is generated by the set of commutators $\{ [g,z] \mid g\in G,\ z\in P,\ [g,z]\in P \}$. In particular, it follows that $P\cap G'$ can be generated by commutators  lying in $P$. This observation led to the question on generation of Sylow subgroups of verbal subgroups of finite groups. More specifically, the following problem was addressed in \cite{focal}.

Given a multilinear commutator word $w$ and a Sylow $p$-subgroup $P$ of  a finite group $G$, is it true that $P\cap w(G)$ can be generated by $w$-values lying in $P$?

The answer to this is still unknown. The main result of \cite{focal} is that if $G$ has order $p^an$, where $n$ is not divisible by $p$,  then $P\cap w(G)$ is generated by $n$th powers of $w$-values. In the present paper we will require a result on generation of Sylow subgroups of $\delta^*_k(G)$.

\begin{lemma}\label{foca}
Let $k\ge 0$ and let $G$ be a finite soluble group of order $p^an$, where $p$ is a prime and $n$ is not divisible by $p$, and let $P$ be a Sylow $p$-subgroup of $G$. Then $P\cap \delta^*_k(G)$ is generated by $n$th powers of $\delta^*_k$-commutators lying in $P$.
\end{lemma}

It seems likely that Lemma \ref{foca} actually holds for all finite groups. In particular, the result in \cite{focal} was proved without the assumption that $G$ is soluble. It seems though that proving Lemma \ref{foca} for arbitrary groups is a complicated task. Indeed, one of the tools used in \cite{focal} is the proof of the Ore Conjecture by M.\,W.\ Liebeck, E.\,A.\ O'Brien, A. Shalev, and P.\,H.\ Tiep \cite{lost} that every element of any finite simple group is a commutator. Recently it was conjectured in \cite{forum} that every element of a finite simple group is a commutator of elements of coprime orders. If this is confirmed, then extending Lemma \ref{foca} to arbitrary groups would be easy. However the conjecture that every element of a finite simple group is a commutator of elements of coprime orders at present is known to be true only for the alternating groups \cite{forum} and the groups ${\rm PSL}(2,q)$ \cite{pellegrini}. Thus, we prove Lemma \ref{foca} only for soluble groups, which is quite adequate for the purposes of the present paper.

Before we embark on the proof of Lemma \ref{foca}, we note a key result from \cite{focal} that we will need.

\begin{lemma} 
\label{Lemma2.2}
Let $G$ be a finite group, and let $P$ be a Sylow $p$-subgroup of $G$. Assume that $N\le L$ are two normal subgroups of $G$, and use the bar notation in the quotient group $G/N$. Let $X$ be a normal subset of $G$ consisting of $p$-elements such that $\overline{P} \cap \overline{L}=\langle \overline{P} \cap \overline{X} \rangle$. Then $P \cap L=\langle P\cap X,P\cap N\rangle$.
\end{lemma}

We are now ready to prove Lemma \ref{foca}.

\begin{proof}
Let $G$ be a counter-example of minimal order. Then $k\ge 1$.

By induction on the order of $G$, the lemma holds for any proper subgroup and any proper quotient of $G$. We observe that $\delta^*_1(G)<G$ since $G$ is not perfect, and by Lemma \ref{Lemma0}, we have $\delta^*_{k+1}(G)=\delta^*_{k}(\delta^*_1(G))$. Since the result holds for $\delta^*_1(G)$, it follows that $P\cap \delta^*_{k+1}(G)$ is generated by $n$th powers of $\delta^*_k$-commutators in $G$. Note that we made use of Remark 3.2 of \cite{focal}.

If $\delta^*_{k+1}(G)\ne 1$, by induction the result holds for $G/\delta^*_{k+1}(G)$.  Combining this with the fact that $P\cap \delta^*_{k+1}(G)$ can be generated by $n$th powers of $\delta^*_k$-commutators, we get a contradiction by Lemma \ref{Lemma2.2}. Hence $\delta^*_{k+1}(G)=1$. Further $\mathcal{O}_{p'}(G)=1$ since $G$ is a minimal counter-example. Therefore $\delta^*_k(G)\subseteq P$, and it is now obvious that $P \cap \delta^*_k (G)$ is generated by $n$th powers of $\delta^*_k$-commutators lying in $P$. So we have our required contradiction.

\end{proof}

\section{Proofs of the main results}

We mention here a needed result of Schur and Wiegold. The much celebrated Schur Theorem states that if $G$ is a group with $|G/Z(G)|$ finite, then $|G'|$ is finite. It is implicit in the work of Schur that if $|G/Z(G)|=m$, then $|G'|$ is $m$-bounded. However, Wiegold produced a shorter proof of this second statement, which also gives the best possible bound. The reader is directed to Robinson (\cite{Robinson}, pages 102-103) for details.

Additionally, for the proof of Theorem \ref{gitis}, we require the following result from \cite{forum}.
\begin{lemma}
\label{lemmaforum}
Let $G$ be a finite group and let $y_1,\ldots,y_k$ be $\delta^*_k$-commutators in $G$. Suppose the elements $y_1,\ldots,y_k$ normalize a subgroup $N$ such that $(|y_i|,|N|)=1$ for every $i=1,\ldots,k$. Then for every $x\in N$ the element $[x,y_1,\ldots,y_k]$ is a $\delta^*_{k+1}$-commutator.
\end{lemma}

Now we are ready to begin.

\begin{proof}[Proof of Theorem \ref{gut}]
Let $X$ be the set of all $\gamma^*_k$-commutators. We wish to show that if  $|X|=m$, then $|\gamma^*_k(G)|$ is $m$-bounded. For convenience we write $K$ for $\langle X\rangle$. Of course, $K=\gamma_\infty (G)$.

The subgroup $C_G(X)$ has index $\le m!$, so $|K/Z(K)|\le m!$ too. By Schur, $K'$ has $m$-bounded order. Therefore, by passing to the quotient, we may assume  $K'=1$, and so $K$ is abelian with $G$ metanilpotent.

It is enough to bound the order of each Sylow subgroup of $K$. We choose a Sylow $p$-subgroup $P$. By passing to the quotient $G/\mathcal{O}_{p'}(G)$, we may assume $K=P$.

By Lemma \ref{Lemma1}, a Hall $p'$-subgroup $H$ of $G$ satisfies $P=[P,_{k-1} H]$.
We know that $P$ is abelian and $P$ is normal in $PH$. 

We denote the set $\{[x,h_1,\ldots,h_{k-1}]: x\in P,h_1,\ldots,h_{k-1}\in H \}$ by $\hat{X}$.

For $x\in P, h_1,\ldots,h_{i-1}\in H$, where $i\ge 2$, we note that $[x,h_1,\ldots,h_{i-1}]$ is a $\gamma^*_i$-commutator. Therefore $\hat{X}\subseteq X$, and $|\hat{X}|\le m$.

By Lemma \ref{Lemma3}, it follows that $|[P,_{k-1}H]|$ is $m$-bounded.
Appealing to Lemma \ref{Lemma1}, we conclude that $|P|$ is $m$-bounded.
\end{proof}

\begin{proof}[Proof of Theorem \ref{gitis}]
Let $X$ be the set of $\delta^*_k$-commutators in $G$. We wish to show here that if $|X|=m$, then $|\delta^*_k(G)|$ is $m$-bounded.
We recall that $\delta^*_k(G)=\gamma_\infty(\delta^*_{k-1}(G))$. For ease of notation we define $Q:=\delta^*_{k-1}(G)$, and we write $K$ for $\delta^*_k(G)$.

The subgroup $C_G(X)$ has index $\le m!$ in $G$, so $|K/Z(K)|\le m!$ and as in the proof of Theorem \ref{gut}, we may assume $K'=1$. Hence $K$ is assumed to be abelian with $Q$ metanilpotent. In what follows, we now restrict to the group $Q$.

It is sufficient to show that the order of each Sylow subgroup of $K$ is $m$-bounded. We choose $P$ a Sylow $p$-subgroup of $K$. By passing to the quotient $G/\mathcal{O}_{p'}(G)$, we may assume $K=P$.

By Lemma \ref{Lemma1}, a Hall $p'$-subgroup $H$ of $Q$ satisfies $P=[P,H]$. By Lemma \ref{foca}, since $H$ is generated by its Sylow subgroups, we have $H$ is generated by a normal subset $B$ of powers of $\delta^*_{k-1}$-commutators that are of $p'$ order.

Lemma \ref{Pavel} now implies that $[P,H]$ is generated by subgroups $[P,b_1,\ldots,b_k]$ for $b_1,\ldots,b_k \in B$. By Lemma \ref{lemmaforum}, for $x\in P$ we have  $[x,b_1,\ldots,b_k]$ is a $\delta^*_k$-commutator, and we deduce that $|[P,b_1,\ldots,b_k]|$ is $m$-bounded.

It follows that the number of generators of $[P,H]$ is at most $m$, and futhermore the exponent of $[P,H]$ is $m$-bounded. Hence, the finite abelian $p$-group $P=[P,H]$ has $m$-bounded order.
\end{proof}


\end{document}